\documentclass[a4paper]{article} 
\usepackage[utf8]{inputenc}
\usepackage{amsmath}
\usepackage[english]{babel}
\usepackage{amsmath}
\usepackage{amssymb}
\usepackage{amsthm}
\usepackage{algorithm}
\usepackage{algorithmic}
\usepackage{caption}
\usepackage{enumerate}
\usepackage{color}
\usepackage{url}

\newlength\myindent
\newtheorem{theorem}{Theorem}
\newtheorem{proposition}[theorem]{Proposition}

\theoremstyle{definition}
\newtheorem{definition}[theorem]{Definition}
\newtheorem{assumption}[theorem]{Assumption}
\newtheorem{example}[theorem]{Example}
\newtheorem{remark}[theorem]{Remark}

\DeclareMathOperator{\gr}{gr}
\DeclareMathOperator{\conv}{conv}
\DeclareMathOperator{\cone}{cone}
\DeclareMathOperator{\vertices}{vert}
\DeclareMathOperator{\dom}{dom}
\DeclareMathOperator{\cl}{cl}
\DeclareMathOperator{\Min}{Min}

\DeclareMathOperator{\Int}{int}

\newcommand{\V}{\mathcal{V}}
\renewcommand{\P}{\mathcal{P}}
\newcommand{\Q}{\mathcal{Q}}
\newcommand{\D}{\mathcal{D}}
\newcommand{\G}{\mathcal{G}}
\newcommand{\R}{\mathbb{R}}
\newcommand*{\colonequals}{\mathrel{\vcenter{\baselineskip0.5ex%
\lineskiplimit0pt\hbox{\scriptsize.}\hbox{\scriptsize.}}}=}

\title{On unbounded polyhedral convex set optimization problems}
\author{Niklas Hey$^1$ \and Andreas Löhne$^2$}

\begin{document}
	
\footnotetext[1]{Vienna University of Economics and Business, Institute for Statistics and Mathematics, Welthandelsplatz 1, 1020 Vienna, Austria, niklas.hey@wu.ac.at}
\footnotetext[2]{Friedrich Schiller University Jena, Faculty of Mathematics and Computer Science, 07737 Jena, Germany, andreas.loehne@uni-jena.de}

\maketitle

\begin{abstract}
	A polyhedral convex set optimization problem is given by a set-valued objective mapping from the $n$-dimensional to the $q$-dimensional Euclidean space whose graph is a convex polyhedron. This problem can be seen as the most elementary subclass of set optimization problems, comparable to linear programming in the framework of optimization with scalar-valued objective function. Polyhedral convex set optimization generalizes both scalar and multi-objective (or vector) linear programming. In contrast to scalar linear programming but likewise to multi-objective linear programming, unbounded problems can indeed have minimizers and provide a rich class of problem instances. In this paper we extend the concept of finite infimizers from multi-objective linear programming to not necessarily bounded polyhedral convex set optimization problems. We show that finite infimizers can be obtained from finite infimizers of a reformulation of the polyhedral convex set optimization problem into a vector linear program. We also discuss two natural extensions of solution concepts based on the complete lattice approach. Surprisingly, the attempt to generalize the solution procedure for bounded polyhedral convex set optimization problems introduced in [A. Löhne and C. Schrage. An algorithm to solve polyhedral convex set optimization problems. {\em Optimization} 62(1):131--141, 2013.] to the case of not necessarily bounded problems uncovers some problems, which will be discussed.
	
\medskip
\noindent
{\bf Keywords:} set optimization, solution methods, vector linear programming, multiple objective linear programming
\medskip

\noindent{\bf MSC 2010 Classification: 90C99, 90C29, 90C05} 
\end{abstract}

\section{Introduction}
A set-valued mapping $F: \mathbb R^n \rightrightarrows \mathbb R^q$ is called {\em polyhedral convex} if its {\em graph}
$$ \gr F \colonequals \{ (x,y) \in \R^n\times \R^q \mid y \in F(x)\}$$
is a convex polyhedron. We minimize a polyhedral convex objective map $F$ with respect to the ordering relation on the power set $2^{\R^q}$ of $\R^q$
\begin{align}\label{qo}
	V \preccurlyeq_C W  :\Longleftrightarrow W \subseteq V+C,
\end{align}
for sets $V,W \in \R^q$ and $C$ being a polyhedral convex cone that is pointed and has nonempty interior. This problem is called a {\em polyhedral convex set optimization problem} and is expressed by
\begin{equation}\label{P}
	\text{minimize} \hspace{0.1cm} F(x) \hspace{0.2cm} \text{with respect to} \preccurlyeq_C \hspace{0.2cm} \text{subject to} \hspace{0.2cm} x \in \R^n.
	\tag{P}
\end{equation}
Explicit linear constraints could be added. The resulting problem, however, is of the same type: Indeed let $H:\R^n\rightrightarrows \R^q$ be a polyhedral convex set-valued mapping to be minimized under the linear constraints $A x \leq b$, for $A \in \R^{m \times n}$ and $b \in \R^m$. Then we can define $F$ by its graph as
$$  \gr F \colonequals \{(x,y) \in \R^n\times \R^q \mid (x,y) \in \gr H,\; Ax \leq b\}.$$
Then, $F$ is polyhedral convex and problem \eqref{P} is equivalent to the given constrained problem. If $H$ is defined as $\gr H = \{(x,y) \in \R^n \times \R^q \mid y = P x\}$ for $P\in \R^{q\times n}$, problem \eqref{P} reduces to a vector linear program. For $C=\R^q_+$ we obtain a multiple objective linear program. and for $q=1$ a linear program. 

A solution concept and a solution procedure for a polyhedral convex set optimization problem being {\em bounded} in the sense of
\begin{align}{\label{boundednessrest}}
\exists v \in \mathbb R^q: \{v\} \preccurlyeq_C \bigcup_{x \in \mathbb R^n} F(x).
\end{align}
has been established in \cite{paper,paper1}. The singleton set $\{v\}$ can be seen as a (special form of) lower bound in the space $2^{\R^q}$ equipped with the {\em pre-ordering} (i.e. reflexive and transitive ordering) $\preccurlyeq_C$, while the union of the right hand side in \eqref{boundednessrest} is an {\em infimum} (greatest lower bound) of $F$ over $x \in \R^q$ in this space.

In the present article, we extend the solution concept and the solution procedure introduced in \cite{paper,paper1} to (almost) arbitrary polyhedral convex set optimization problems, in particular, to unbounded problems. 
For the special case of vector linear programming, both a solution concept and a solution procedure for not necessarily bounded problems have been established in \cite{loehne2011}.

Throughout this article we suppose for problem \eqref{P}:
\begin{assumption}\label{ass1}
	The set $\P \colonequals \bigcup_{x \in \R^n} F(x)+C$, called {\em upper image of \eqref{P}}, is free of lines.
\end{assumption} 
In \cite{loehne2011}, for the special case of vector linear programs, the upper image was assumed to have a vertex. This is conform with Assumption \ref{ass1} because $\P$ is a convex polyhedron, as shown below in Section \ref{sec_3}, and a convex polyhedron is free of lines if and only if it has a vertex. Recently, for vector linear programs this assumption was dropped in \cite{weissing20}.

The solution concept discussed and extended here is based on the complete lattice approach to set optimization, see e.g. \cite{rathershortintr} for historical remarks and other approaches. 
The space $$\P \colonequals \{ V \subseteq \R^q \mid V = V+C\}$$
equipped with the ordering defined in \eqref{qo}, which coincides with $\supseteq$ on $\P$ and hence is a partial ordering on $\P$ (i.e. reflexive, transitive and antisymmetric), provides a {\em complete lattice}. This means that for every subset of $\P$ an infimum in the sense of the (uniquely defined) greatest lower bound exists. The infimum of a subset $\V \subseteq \P$ is given by 
$$ \inf \V = \bigcup_{V\in \V} V.$$
A solution concept for the complete lattice approach in set optimization was introduced in \cite{heydeloehne11} (contributed by A.H. Hamel). 
It is typical for this approach that, in contrast to scalar optimization, {\em infimum attainment} and {\em minimality} are different conditions. A {\em solution} in the sense of the {\em complete lattice approach} satisfies both {\em infimum attainment} and {\em minimality}.
Likewise to solutions for unbounded vector linear programs \cite{loehne2011}, a solution of an unbounded polyhedral convex set optimization problem is expected to consist of finitely many feasible points and finitely many feasible directions. The latter part is new in comparison to the bounded case in \cite{paper,paper1}. 
We discuss two possible extensions of this solution concept with two different minimality notions.

The solution procedure for bounded polyhedral set optimization problems introduced in \cite{paper,paper1} consists of two phases: In the first phase a vector linear program is solved to obtain infimum attainment, i.e., to compute an infimizer. In the second phase finitely many linear programs (combined with vertex enumeration in the objective space $\R^q$) are solved to ensure minimality, i.e., to ensure that the elements of the infimizer are minimizers. 

In this article we introduce the concept of an infimizer for a not necessarily bounded polyhedral convex set optimization problem. 
We show that such an infimizer can be obtained, likewise to \cite{paper,paper1}, by a reformulation of the problem into a vector linear program. Moreover we discuss two types of minimizers, which lead to two different solution concepts. For both notions we did not succeed with a natural generalization of the second phase of the solution method for the bounded case from \cite{paper,paper1}. The one solution concept destroys the convexity of the problem while the other one seems to require nonlinear scalar problems for the solution method. Furthermore, we slightly extend some results in \cite{paper,paper1} for the bounded case as we allow a more general representation of the graph of $F$.

Applications for set optimization can be found in mathematical finance, when we look at markets with frictions, see e.g. \cite{hamelheyde10,riskmeasure,loehnerudloff14}. Solution concepts in set optimization are also useful to provide an approach to vector optimization based on infimum attainment, see e.g. \cite{heydeloehne11,loehne2011,loehnetammer07}. Other applications for set-valued optimization appear in multivariate statistics \cite{hamelkostner18}. We want to remark that the solution concepts of the complete lattice approach are frequently used in theoretical papers while, so far,  the focus of the mentioned applications is on infimum attainment only. It remains an interesting task for the future to point out the impact of minimality for applications. In view of the mentioned problems in the unbounded case it could be worth to discuss whether or not minimality should be replaced by another property.

\section{Preliminaries}

A set $P \subseteq \mathbb R^q$ is called {\em polyhedral convex} or a {\em convex polyhedron} if it can be expressed as
\begin{align}{\label{hrep0}}
	P=\{y \in \mathbb R^q \mid By \geq b\},
\end{align}
for $B \in \mathbb R^{m \times q}$ and $b \in \mathbb R^m$. The representation of $P$ in \eqref{hrep0} is called an {\em H-representation}. 
The {\em recession cone} of $P$ in \eqref{hrep0} is the set
$$0^{+}P\colonequals \{y \in \mathbb R^q \mid By \geq 0\}.$$
 A convex polyhedron $P$ is {\em bounded} if and only if $0^{+}P=\{0\}$. A bounded polyhedral convex set is called a {\em polytope}.
 A convex polyhedron $P$ can also be written as the convex hull of finitely many points $v_1,...,v_k \in \mathbb R^q$ with $k \in \mathbb N$ and the conic hull (the smallest convex cone containing a given set) of finitely many directions $d_1,...,d_l \in \mathbb R^q$ with $l \in \mathbb N_0$ in the form
\begin{align}{\label{vrep0}}
    P=\conv\{v_1,...,v_k\}+\cone\{d_1,...,d_l\}.
\end{align}
A representation of $P$ as in \eqref{vrep0} is called a {\em{V-representation}}. Note that we use the convention $\cone \emptyset=\{0\}$ in order to represent polytopes without any direction. We have
$$ 0^{+}P = \cone\{d_1,...,d_l\}.$$ 

Let $P$ be a convex polyhedron. A point $v$ of $P$ is called a {\em vertex} of $P$ if there do not exist $v_1,v_2 \in P$ with $v_1 \neq v_2$ such that $v=\lambda v_1+(1-\lambda)v_2$ for some $\lambda \in (0,1)$. The set of vertices of $P$ is denoted by $\vertices P$. An element $d \in 0^{+}P$ is called an {\em extremal direction} of $P$ if $d \neq 0$ and for $u,w \in 0^{+}P$ with $d=u+w$ we have $u,w \in \cone\{d\}$. $P$ is called {\em pointed} if it has a vertex. $P$ is pointed if and only if $P$ contains no line (see e.g. \cite{lauritzen}, Theorem 4.7). If $P$ is {\em pointed} then there exists a V-representation of $P$ in the form \eqref{vrep0} with $v_1,...,v_k$ being the vertices and $d_1,...,d_l$ being the extremal directions of $P$.

We say the expression 
\begin{align}{\label{prep}}
	P=\{y \in \mathbb R^q \mid \exists x\in \mathbb R^n:\; Ax+By \geq b\}
\end{align}
is a {\em P-representation} (here $P$ stands for {\em projection}) of a convex polyhedron $P$. Using the Fourier-Motzkin elimination (see e.g. \cite{lauritzen}) for the variables $x_1,...,x_n$ in \eqref{prep} we obtain an H-representation of $P$, showing that \eqref{prep} indeed describes a convex polyhedron. An H-representation is clearly a special case of a P-representation. If $V$ is the matrix with columns $v_1,...,v_k$ and $D$ is the matrix with columns $d_1,\dots,d_l$, then 
$$ P = \{ y \in \R^q \mid \exists (\lambda,\mu) \in \R^k \times \R^l:\; V \lambda + D \mu = y,\; \lambda \geq 0,\; e^T \lambda = 1,\; \mu \geq 0\},$$
where $e^T=(1,\dots,1)$, shows that the V-representation in \eqref{vrep0} is also a special case of a P-representation. 

If $Q \subseteq \R^n \times \R^q$ is a nonempty convex polyhedron and 
	$$ P\colonequals \{y \in \R^q \mid \exists x \in \R^n:\; (x,y) \in Q\},$$ 
	then
	\begin{equation}\label{rec_prep}
		 0^+P \colonequals \{y \in \R^q \mid \exists x \in \R^n:\;  (x,y) \in 0^+Q\}.
	\end{equation}	
This statement can be easily shown by using a V-representation of $Q$.
Thus, if $P$ in \eqref{prep} is nonempty, its recession cone is
\begin{equation*}
	0^{+}P=\{y \in \mathbb R^q \mid \exists x \in \mathbb R^n: Ax+By \geq 0\}.
\end{equation*}

 For two convex polyhedra $P_1,P_2 \subseteq \mathbb R^q$ the Minkowski sum $P_1+P_2$ is a convex polyhedron (see e.g. \cite[Corollary 19.3.2]{rockafellar}). The recession cone of the Minkowski sum of two convex polyhedra $P_1$ and $P_2$ can be expressed as 
\begin{equation}\label{eq_recc}
0^{+}(P_1+P_2)=0^{+}P_1+0^{+}P_2,	
\end{equation}
which can be seen by using V-representations.
The Cartesian product $P_1 \times P_2$ is again a convex polyhedron. This becomes obvious when working with H-representations.

Let $C \subseteq \mathbb R^q$ be a pointed polyhedral convex cone with nonempty interior. $C$ can be expressed by some matrix $Z \in \mathbb R^{q \times p}$ as
\begin{align}{\label{coneChrep}}
	C=\{y \in \mathbb R^q \mid Z^Ty \geq 0\}.
\end{align}
By $\Int C \neq \emptyset$ there exists $y \in C$ with $Z^T y>0$. The polyhedral convex cone $C$ in \eqref{coneChrep} is pointed if and only if the matrix $Z$ has rank $q$. $C$ defines a partial ordering $\leq_C$ on $\mathbb R^q$ by 
\begin{align}{\label{vectorordering}}
    v \leq_C w :\Longleftrightarrow w-v \in C.
\end{align}
We say $\bar{z} \in P$ is {\em{C-minimal}} in a set $P\subseteq \R^q$ if there is no $z \in P$ with $z \neq \bar{z}$ and $z \leq_C \bar{z}$. By $\Min_C P$ we denote the set of all $C$-minimal points of $P$. The condition $z \in \Min_C P$ is equivalent to $z \in P$ and $z \notin P+C\!\setminus\!\{0\}$.

The ordering relation $\preccurlyeq_C$ on the power set $2^{\R^q}$ of $\R^q$ is a pre-order, i.e., it is reflexive and transitive. It is also a pre-order in the subspace $\G$ of all closed convex subsets of $\R^q$. For any subset $\D$ of $\G$ an infimum exists and can be expressed as
(see e.g. \cite{rathershortintr}, \cite{loehne2011})
\begin{align}{\label{infimumallgemein}}
		\inf \mathcal{D}=\cl \conv \bigcup_{D \in \mathcal{D}} (D+C).
\end{align}
We use here the convention $\inf \emptyset=\emptyset$.

Let $F: \mathbb R^n \rightrightarrows \mathbb R^m $ be a set-valued map. The {\em domain} of $F$ is the set $\dom F=\{x \in \mathbb R^n \mid F(x) \neq \emptyset\}$. The {\em{graph}} of $F$ is defined as	$\gr F\colonequals\{(x,y) \in \mathbb R^{n \times q} \mid y \in F(x)\}$.
We say a set-valued map $F$ is {\em{polyhedral convex}} if the graph of $F$ is a convex polyhedron. For $y \in \mathbb R^q$ we denote by
$F^{-1}(y)\colonequals\{x \in \mathbb R^n \mid (x,y) \in \gr F\}$ the {\em{inverse}} of $F$ at $y$.
We define $G(x)\colonequals(0^{+}F)(x)$ as the {\em recession mapping} of $F$, which is defined by $\gr G = 0^{+}\gr F$.

\begin{proposition}{\label{constantrecessioncone}}
Let $F: \mathbb R^n \rightrightarrows \mathbb R^q$ be a polyhedral convex set-valued map. Then the mapping $x \mapsto 0^{+}[F(x)]$ is constant on $\dom F$. Moreover, for all $x \in \dom F$ and all $u \in \dom G$ we have $0^+[F(x)] = 0^+[G(u)] = G(0)$.
\end{proposition}
\begin{proof}
The graph of $F$ has an H-representation 
$$\gr F = \{(x,y) \in \R^n\times \R^q \mid A x + B y \geq b\}$$
 and thus  
$F(x) = \{y \in \R^q \mid B y \geq b-Ax\}$. For $x \in \dom F$ we obtain $0^{+}[F(x)]=\{y \in \R^q \mid B y \geq 0\}$, which is independent of $x$.
The recession mapping $G$ is obtained by setting $b=0$, which yields the remaining statements. 
\end{proof}

\begin{proposition}{\label{infpolyhedron}} 
Let $F: \mathbb R^n \rightrightarrows \mathbb R^q$ be a polyhedral convex set-valued map with $\dom F \neq \emptyset$ Then the set $\bigcup_{x \in \mathbb R^n} F(x)$ is a convex polyhedron and
\begin{align*}
	0^{+}\bigcup_{x \in \mathbb R^n} F(x)=\bigcup_{x \in \mathbb R^n} G(x).
\end{align*}
Moreover, we have $0^+(\dom F)=\dom G$.
\end{proposition}
\begin{proof}
We have
\begin{equation*}\label{pr1}
	\bigcup_{x \in \mathbb R^n} F(x) = \{ y \in \R^q \mid \exists x \in \R^n:\; (x,y) \in \gr F\}
\end{equation*}
Since $\gr F$ is a convex polyhedron, this is a P-representation and thus a convex polyhedron. 
The second statement follows from \eqref{rec_prep}.
We have 
$$ \dom F = \{x \in \R^n \mid \exists y \in \R^q: (x,y) \in \gr F\},$$
which implies the last statement is a similar way.
\end{proof}

\begin{proposition}\label{prop_poly_conv}
	Let $P_i \subseteq \R^q$, $i \in \{1,\ldots,k\}$ be finitely many convex polyhedra such that $0^+ P_i = D$ for all $i \in \{1,\ldots,k\}$. Then
	$$ \conv \bigcup_{i=1}^k P_i \qquad \text{and} \qquad \cone \bigcup_{i=1}^k P_i + D $$
	are convex polyhedra.
	Moreover, we have
	\begin{equation*}	
			0^+ \conv \bigcup_{i=1}^k P_i = D
	\end{equation*}
	and
	\begin{equation}\label{eq22}	
			\cl\cone \bigcup_{i=1}^k P_i = \cone \bigcup_{i=1}^k P_i + D.
	\end{equation}
\end{proposition}	
\begin{proof}
	There are polytopes $Q_i \subseteq \R^q$ such that $P_i = Q_i + D$ for all $i \in \{1,\ldots,k\}$. The set $Q\colonequals \conv \bigcup_{i=1}^k Q_i$ is a polytope.
	Thus 
	$$P \colonequals \conv \bigcup_{i=1}^k P_i = \conv \bigcup_{i=1}^k [Q_i + D] = Q + D$$
	is a convex polyhedron with recession cone $D$.	We have 
	\begin{align*}
		\cone \bigcup_{i=1}^k P_i + D &= \cone P + D = \cone \left[Q + D\right] + D 
									  = \cone Q + D,
	\end{align*}
	which is a convex polyhedron because the conic hull of the polytope $Q$ is a convex polyhedron.
	
	To show \eqref{eq22}, recall that $\cl\cone P = \cone P \cup 0^+ P$ holds, which is a consequence of  \cite[Theorem 8.2]{rockafellar}. Thus 
	\begin{equation*}
		\cl\cone \bigcup_{i=1}^k P_i = \cl\cone P = \cone P \cup D = \cone P + D = \cone \bigcup_{i=1}^k P_i + D. \qedhere
	\end{equation*}
\end{proof}

Note that the condition $0^+ P_i = D$ cannot be omitted in the previous result, which can be seen by the example
$$ \conv \left(\R^2_+ + \{(1,1)^T\} \cup \{(0,0)^T\} \right) = \Int \R^2_+ \cup \{(0,0)^T\}.$$
Moreover, $\cone \bigcup_{i=1}^k P_i$ is not necessarily a convex polyhedron, not even for $k=1$. For example, 
$$ \cone (\R^2_+ + \{(1,1)^T\}) = \Int \R^2_+ \cup \{(0,0)^T\}.$$

\section{Finite infimizers}\label{sec_3}

In this section we extend the concept of a finite infimizer to the case of {\em unbounded} polyhedral convex set-valued optimization problems \eqref{P}. The new concept extends both finite infimizers for bounded polyhedral convex set-valued optimization problems \cite{paper,paper1} and finite infimizers for not necessarily bounded vector linear programs \cite{loehne2011}. 

\begin{definition}\label{def_inf}
A tuple $(\bar{X},\hat{X}$) of finite sets $\bar{X} \subseteq \dom F$, $\bar{X} \neq \emptyset$ and $\hat{X} \subseteq \dom G \setminus \{0\}$ is called a {\em finite infimizer} for problem \eqref{P} if 
\begin{align}{\label{finiteinfimizer}}
	\bigcup_{x \in \mathbb R^n} F(x) \subseteq \conv \bigcup_{x \in \Bar{X}} F(x) + \cone \bigcup_{x \in \hat{X}} G(x) +C
\end{align}
\end{definition}

In this definition, \eqref{finiteinfimizer} can be replaced by
\begin{align}{\label{finiteinfimizer2}}
	\bigcup_{x \in \mathbb R^n} F(x)+C = \conv \bigcup_{x \in \Bar{X}} F(x) + \cone \bigcup_{x \in \hat{X}} G(x) +C,
\end{align}
which can be seen by using Proposition \ref{infpolyhedron}.
This shows that a finite infimizer generates the upper image $\P$ of \eqref{P}. By Proposition \ref{infpolyhedron}, both 
the {\em image} $\Q\colonequals \bigcup_{x \in \mathbb R^n} F(x)$ and the {\em upper image} 
$$\P \colonequals \Q+C = \bigcup_{x \in \mathbb R^n} (F(x)+C)$$
of \eqref{P} are convex polyhedra. From \eqref{eq_recc} and Proposition \ref{constantrecessioncone} we obtain
$$0^{+} \P=0^{+}\Q+C$$
and 
$$ \forall x \in \dom F:\quad 0^{+}[F(x)+C]=0^{+}[F(x)]+C = G(0) + C.$$

\begin{remark}\label{rem0} The concept of an {\em upper image} of \eqref{P} corresponds to an infimum with respect to the ordering $\preccurlyeq_C$ in the space $\G$, compare \eqref{infimumallgemein}. By Proposition~\ref{infpolyhedron}, the infimum can be expressed without the closure and the convex hull. The term {\em upper image} is often used in the framework of vector optimization, when the property of being an infimum in the space $\G$ is not addressed.
\end{remark}

Even though the convex hull of the union of finitely many convex polyhedra is not necessarily closed, the expression $\conv \bigcup_{x \in \bar X} F(x)$ in Definition \ref{def_inf} is a convex polyhedron, compare Proposition \ref{prop_poly_conv}.
Note that the term $\cone \bigcup_{x \in \hat X} G(x)$ in Definition \ref{def_inf} is not necessarily closed, compare the example after Proposition \ref{prop_poly_conv}.
Nevertheless the term
\begin{equation}\label{eq99}
	\conv \bigcup_{x \in \Bar{X}} F(x) + \cone \bigcup_{x \in \hat{X}} G(x)
\end{equation}
in Definition~\ref{def_inf} is always a convex polyhedron. We also show in the next proposition that
in Definition \ref{def_inf} the exclusion of zero in the set $\hat X$ is not restricting.

\begin{proposition}\label{no_zero} Let $(\bar{X},\hat{X}$) be a tuple of finite sets $\bar{X} \subseteq \dom F$, $\bar{X} \neq \emptyset$ and $\hat{X} \subseteq \dom G$. Then $(\bar X, \hat X)$ satisfies \eqref{finiteinfimizer} if and only if $(\bar X, \hat X \setminus \{0\})$ satisfies \eqref{finiteinfimizer}. Moreover, the expression in \eqref{eq99} is a convex polyhedron.
\end{proposition}
\begin{proof}
We set 
$$ P\colonequals \conv\bigcup_{x \in \bar X} F(x),\quad  Q\colonequals \cone \bigcup_{x \in \hat X\setminus\{0\}} G(x),\quad  Q_0\colonequals \cone \bigcup_{x \in \hat X} G(x).$$
Propositions \ref{constantrecessioncone} and \ref{prop_poly_conv} yield $P = P + G(0)$. Moreover, we have $Q+G(0)=Q_0+G(0)$. Thus $P + Q = P + G(0) + Q = P + G(0) + Q_0 = P + Q_0$, which yields the first claim. 
By Propositions \ref{constantrecessioncone} and \ref{prop_poly_conv}, $P$ and $G(0)+Q_0$ are convex polyhedra, hence $P+Q$ is a convex polyhedron.
\end{proof}

Note that in Definition \ref{def_inf} the set $\hat{X}$ is allowed to be empty. In this case, problem \eqref{P} is bounded and we obtain the solution concept introduced in \cite{paper,paper1} as a particular case. If the graph of $F$ is of the form 
$$  \gr F \colonequals \{(x,y) \in \R^n\times \R^q \mid y = Px,\; Ax \leq b\},$$
then Definition \ref{def_inf} describes a finite infimizer (as introduced in \cite{loehne2011}) for the vector linear program
\begin{equation}\label{VLP}
	\text{minimize} \hspace{0.1cm} P x \hspace{0.3cm} \text{with respect to}  \hspace{0.2cm} \leq_C \hspace{0.2cm} \text{subject to} \hspace{0.2cm} A x \leq b,
	\tag{VLP}
\end{equation}
where $P \in \R^{q \times n}$, $A \in \R^{m \times n}$, $b \in \R^m$ and $C$ is defined as in \eqref{coneChrep}. For $q=1$, $\bar X$ can be chosen as a singleton set and $\hat X$ can be chosen to be empty. This is the particular case of a linear program.

As a consequence of Proposition \ref{infpolyhedron}, we have
\begin{equation}\label{eq_33}
	 0^+\P = \bigcup_{x \in \mathbb R^n}G(x)+C,
\end{equation}
This shows that elements in $0^+ \P$ not belonging to $C$ can be expressed by values of the recession mapping $G$. Such elements occur exactly in the unbounded case. This motivates the usage of the recession mapping in Definition \ref{def_inf}, where the concept of a finite infimizer is extended to not necessarily bounded problems.

The next proposition provides a recursive definition of a finite infimizer. In particular, it is shown that $0^+\P$ is generated by the direction part $\hat X$ of a finite infimizer.

\begin{proposition}
	A tuple $(\bar{X},\hat{X}$) of finite sets $\bar{X} \subseteq \dom F$, $\bar{X} \neq \emptyset$ and $\hat{X} \subseteq \dom G \setminus \{0\}$, $\hat{X} \neq \emptyset$ is a finite infimizer for problem \eqref{P} if and only if	
\begin{equation}\label{finiteinfimizer3}
	\P = \conv \bigcup_{x \in \bar X} F(x) + 0^+\P \qquad\text{and}\qquad 0^+\P = \cl\cone \bigcup_{x \in \hat X} G(x) + C.
\end{equation}
\end{proposition}
\begin{proof} 
	We set 
	$$Q\colonequals \conv \bigcup_{x \in \bar X} F(x) \quad\text{ and }\quad R\colonequals \cone \bigcup_{x \in \hat X} G(x).$$
	Then condition \eqref{finiteinfimizer2}, which is shown to be equivalent to \eqref{finiteinfimizer3}, can be written as $\P = Q + R + C$. By Proposition~\ref{no_zero}, \eqref{finiteinfimizer2} is equivalent to $\P = Q + \cl R + C$. Thus \eqref{finiteinfimizer3} implies \eqref{finiteinfimizer2}.
	
	Now let $\P = Q + \cl R + C$ be satisfied. Since $\hat X \neq \emptyset$, Propositions \ref{constantrecessioncone} and \ref{prop_poly_conv} yield  $\cl R = R+G(0)$ and $0^+ Q = G(0)$. Since $\cl R$ and $C$ are polyhedral convex cones, we obtain $0^+(\cl R +C) = \cl R + C \supseteq  G(0) = 0^+Q$. From \eqref{eq_recc} we obtain $0^+\P = 0^+ Q + 0^+ (\cl R+C)$. We deduce $0^+\P = \cl R + C$ and $\P = Q + 0^+\P$, i.e., \eqref{finiteinfimizer3} holds.
\end{proof}

\section{Existence and computation of finite infimizers}

The goal of this section is to show that a finite infimizer for the unbounded problem (P) always exists if the objective function is polyhedral convex and the upper image (or infimium, compare Remark \ref{rem0}) $\P$ contains no line. We also show that a finite infimizer can be computed by solving a {\em polyhedral projection problem}, or alternatively, a {\em vector linear program}. Note that these problems are equivalent, see \cite{pp}. 

Let the graph of $F$ be given by a P-representation
\begin{align}{\label{prepgr}}
    \gr F=\{(x,y) \in \mathbb R^{n \times q} \mid \exists u \in \mathbb R^p: Ax+By+Qu \geq b\},
\end{align}
where $A \in \mathbb R^{m \times n}$, $B \in \mathbb R^{m \times q}$, $Q \in \mathbb R^{m \times p}$ and $b \in \mathbb R^m$. Of course, such a P-representation always exists. Note that in \cite{paper,paper1} (where in contrast to this article only bounded problems were studied), an H-representation of $\gr F$ was expected to be known. The P-representation \eqref{prepgr} is more general. Thus the results in \cite{paper,paper1} are extended here even for the bounded case.

The graph of the recession mapping $G=0^+F$ of $F$ can be expressed as
\begin{align}{\label{graphG}}
    \gr G=\{(x,y)\in \mathbb R^{n \times q} \mid \exists u \in \mathbb R^p: Ax+By+Qu \geq 0\}.
\end{align}

By Proposition~\ref{infpolyhedron} and by Assumption~\ref{ass1}, $\P=\bigcup_{x \in \mathbb R^n}F(x)+C$ is a pointed convex polyhedron. Thus there is a V-representation
\begin{align}{\label{vreptildep}}
    \P=\conv\{v_1,...,v_k\}+\cone\{d_1,...,d_l\},
\end{align}
where $v_1,...,v_k$ are the vertices and $d_1,...,d_l$ are the extremal directions of $\P$. Without loss of generality we assume that 
$d_1,\ldots,d_t \not\in C$ and $d_{t+1},\ldots,d_l \in C$ for $t \in \{0,\ldots,l\}$.
Then we obtain the representation
\begin{align}{\label{vreptildep1}}
    \P=\conv\{v_1,...,v_k\}+\cone\{d_1,...,d_t\}+C,
\end{align}
We define the sets
\begin{align}
   \label{fininf1} &\bar{X}\colonequals\{\bar{x}_1,...,\bar{x}_k\} \hspace{0.2cm} \textit{with} \hspace{0.2cm} \bar{x}_i \in F^{-1}(v_i), \;  i \in \{1,...,k\}\\
   \label{fininf2} &\hat{X}\colonequals\{\hat{x}_1,...,\hat{x}_t\} \hspace{0.2cm} \textit{with} \hspace{0.2cm} \hat{x}_j \in G^{-1}(d_j), \;  j \in \{1,...,t\}. 
\end{align}
We prove that the sets $\bar{X}$, $\hat{X}$ are well-defined and that the tuple $(\bar{X},\hat{X}\setminus\{0\})$ provides a finite infimizer for problem \eqref{P}. 

\begin{proposition}{\label{minpointlemma}}
There exist sets $\bar X$ and $\hat X$ satisfying \eqref{fininf1} and \eqref{fininf2}.
\end{proposition}
\begin{proof}
	(i) Let $v$ be a vertex of $\P$. Then $v=y+c$ for $y \in \bigcup_{x \in \R^n}F(x)$ and $c \in C$. If $c \neq 0$, $v$ is a nontrivial convex combination of $y \in \P$ and $y+2c \in \P$, which contradicts $v$ being a vertex of $\P$. Thus we have $v \in \bigcup_{x \in \R^n}F(x)$ which implies $F^{-1}(v) \neq \emptyset$. 

(ii) Let $d \in 0^+\P \setminus C$ be an extremal direction of $\P$. By \eqref{eq_33} we obtain $d=y+c$ for 
$y \in \bigcup_{x \in \R^n} G(x) \subseteq 0^+\P$ and $c \in C \subseteq 0^+\P$. From the definition of an extremal direction we obtain $c \in \cone\{d\}$. Since $d \notin C$ we conclude $c=0$. Thus we have $d \in \bigcup_{x \in \R^n}G(x)$ which implies $G^{-1}(d) \neq \emptyset$. 
\end{proof}

\begin{theorem}{\label{existencefiniteinfimizer}}
Let Assumption \ref{ass1} be fulfilled for problem \eqref{P}. Then the tuple $(\bar{X},\hat{X}\setminus\{0\})$ as defined in \eqref{fininf1} and \eqref{fininf2} is a finite infimizer for \eqref{P}.
\end{theorem}
\begin{proof}
By Assumption \ref{ass1}, $\P$ has a vertex. Thus we have $\bar X \neq \emptyset$. By \eqref{vreptildep1}, \eqref{fininf1} and \eqref{fininf2} we obtain that \eqref{finiteinfimizer} holds for $\bar X$ and $\hat X$. The claim now follows from Proposition \ref{no_zero}.
\end{proof}


If the vertices and extremal directions of $\P$ are known, a finite infimizer for \eqref{P} can be obtained by solving finitely many linear programs. For instance, for $v$ running over the vertices of $\P$, the finitely many feasibility problems
\begin{align}{\label{fp}}
 \min_{x \in \mathbb R^n} 0^Tx \hspace{0.2cm} \textit{s.t.} \hspace{0.2cm} x \in F^{-1}(v)
\end{align}
yield the set $\bar X$. Since the condition $x \in F^{-1}(v)$ in \eqref{fp} can be written equivalently as $Ax+Qu \geq b-Bv$ with an auxiliary variable $u \in \mathbb R^p$, \eqref{fp} is a linear program with variables $x$ and $u$. The elements of $\hat X$ can be obtained analogously by using the mapping $G$ instead of $F$.

In \cite{paper,paper1}, a finite infimizer for the bounded polyhedral convex set-valued optimization problem was obtained by solving an associated vector linear program. This approach also works for unbounded problems.
 
Consider the projection $f: \mathbb R^{n \times q} \to \mathbb R^q$ with $f(x,y) \colonequals y$. The function $f$ can be understood as a set-valued map whose values are singleton sets, i.e.
\begin{align*}
	f: \mathbb R^{n \times q} \rightrightarrows \mathbb R^q, \hspace{0.3cm} f(x,y)=\{y\}.
\end{align*}
As in \cite{paper,paper1} we define the {\em vectorial relaxation} of \eqref{P} as
\begin{align*}
	\min_{x,y} f(x,y) \hspace{0.2cm} \text{w.r.t.} \leq_C \hspace{0.2cm} \text{s.t.} \hspace{0.2cm} y \in F(x) \tag{VR}\label{VR}
\end{align*}
A finite infimizer for the vector linear program \eqref{VR} coincides with a finite infimizer in the sense of Definition \ref{def_inf} above of the polyhedral convex set optimization problem \eqref{P} if the objective mapping $F$ in \eqref{P} is replaced by $\tilde F:\R^n \times \R^q \rightrightarrows \R^q$, defined by
$$ \gr \tilde F \colonequals \{ (x,y,w) \in \R^n\times\R^q\times\R^q \mid w =y,\; (x,y) \in \gr F \}.$$
We have $\tilde F(x,y)=\{y\}$ if $(x,y)\in \gr F$ and $\tilde F(x,y)=\emptyset$ otherwise. Thus
\begin{align}{\label{VRinfimum}}
  \bigcup_{x \in \R^n} F(x)+C = \bigcup_{x \in \mathbb R^n \atop y \in F(x)} \{y\}+C=\bigcup_{x \in \R^n \atop y \in \R^q} \tilde F(x,y) + C.
\end{align}
Of course the graph of the recession mapping $\tilde G \colonequals 0^+ \tilde F$ is 
$$ \gr \tilde G \colonequals \{ (x,y,w) \in \R^n\times\R^q\times\R^q \mid w =y,\; (x,y) \in \gr G \}.$$
We have $\tilde G(x,y)=\{y\}$ if $(x,y)\in \gr G$ and $\tilde G(x,y)=\emptyset$ otherwise.

\begin{theorem}{\label{VRsolutioninfimizer}}
Let $(\bar{Z},\hat Z)$ with 
\begin{align*}
    \bar{Z}=\bigg\{\begin{pmatrix} \bar{x}_1\\ \bar{y}_1\end{pmatrix},\ldots,\begin{pmatrix} \bar{x}_k \\ \bar{y}_k\end{pmatrix}\bigg\} \subseteq \R^n \times \R^q,
\end{align*}
\begin{align*}
    \hat{Z}=\bigg\{\begin{pmatrix} \hat{x}_1\\ \hat{y}_1\end{pmatrix},\ldots,\begin{pmatrix} \hat{x}_l \\ \hat{y}_l\end{pmatrix}\bigg\} \subseteq (\R^n \times \R^q) \setminus \{0\}.
\end{align*}
be a finite infimizer for the vectorial relaxation \eqref{VR} of \eqref{P}.
Then $(\bar X,\hat X)$ with  
$$\bar{X}\colonequals\{\bar{x}_1,\ldots,\bar{x}_k\} \quad \text{and} \quad \hat{X} \colonequals \{\hat{x}_1,\ldots,\hat{x}_l\}\setminus \{0\}$$
is a finite infimizer for problem (P).
\end{theorem}

\begin{proof}
Let $(\bar{Z},\hat{Z})$ be a finite infimizer for \eqref{VR}. Thus $\bar Z \neq \emptyset$, $\bar Z \subseteq \dom \tilde F$ and $\hat Z \subseteq \dom \tilde G\setminus \{0\}$. For $(\bar x_i,\bar y_i) \in \bar Z$, we have $\tilde F(\bar x_i,\bar y_i)= \{\bar y_i\}$ and for $(\hat x_j,\hat y_j) \in \hat Z$, we have $\tilde G(\hat x_j,\hat y_j)= \{\hat y_j\}$. From \eqref{finiteinfimizer2} we conclude
$$ \bigcup_{x \in \R^n \atop y \in \R^q} \tilde F(x,y) + C = \conv\{\bar{y}_1,\ldots,\bar{y}_k\}+\cone\{\hat{y}_1,\ldots,\hat{y}_l\}+C.$$
It is easy to verify that $\bar X \neq \emptyset$, $\bar X \subseteq \dom F$ and $\hat X \subseteq \dom G\setminus \{0\}$.
Using \eqref{VRinfimum} and the facts that $\bar y_i \in F(\bar x_i)$, $i \in \{1,\ldots,k\}$ and $\hat y_j \in G(\hat x_j)$, $j \in \{1,\ldots,l\}$ we obtain
\begin{align*}
    \bigcup_{x \in \mathbb R^n} F(x) = \conv \bigcup_{i=1}^k F(\bar x_i) + \cone \bigcup_{j=1}^l G(\hat x_j)+C.
\end{align*}
By Proposition \ref{no_zero}, $(\bar X,\hat X)$ is a finite infimizer for \eqref{P}.
\end{proof}

We conclude that a finite infimizer for problem (P) can be obtained by solving the vector linear program \eqref{VR} and thus can be realized by a vector linear program solver like {\em bensolve} \cite{bensolve1, bensolve2}. We close this section by a remark.

\begin{remark}\label{rem4} The objective mapping $F:\R^n\rightrightarrows \R^q$ can be replaced by 
	$$ F_C :\R^n\rightrightarrows \R^q,\quad F_C(x)\colonequals F(x)+C.$$
	A P-representation of $\gr F_C$ can be easily obtained from a P-representation of $\gr F$ as we have 
	$$ \gr F_C = \gr F + (\{0\}\times C).$$	
	By \eqref{eq_recc}, we obtain the recession mapping $G_C$ of $F_C$ as
	$$ G_C :\R^n\rightrightarrows \R^q,\quad G_C(x)\colonequals G(x)+C.$$ 
	The concepts and results of this exposition can be reformulated by replacing 
 $F$ $G$ and $\preccurlyeq_C$ by $F_C$, $G_C$ and $\supseteq$, respectively.	
\end{remark}

\section{Solution concepts}

In order to make an infimizer a {\em solution} of a polyhedral convex set optimization problem, we need to define minimizers. Since an infimizer is a tuple of a set of points and a set of directions, we introduce likewise two types of minimizers: {\em minimizing points} and {\em minimizing directions}. 
We discuss here three variants of solution concepts. 

The first variant is in all aspects based on the ordering relation induced by the cone $C$.

\begin{definition}\label{def_min_sol_0}
	An element $\bar{x} \in \dom F$ is called a {\em minimizing point} for \eqref{P} if 
	\begin{align*}
		F(x) \preccurlyeq_C F(\bar x) \quad \Rightarrow \quad F(x) + C = F(\bar x) + C.
	\end{align*}
	A nonzero element $\hat{x} \in \dom G$ is called a {\em minimizing direction} for \eqref{P} if
	\begin{align*}
		G(x) \preccurlyeq_C G(\hat{x}) \quad \Rightarrow \quad G(x) + C = G(\hat{x}) +  C.
	\end{align*}
	A finite infimizer $(\bar{X},\hat{X})$ is called a {\em solution} to \eqref{P} if all elements of  $\bar{X}$ are minimizing points and all elements of $\hat{X}$ (if any) are minimizing directions.
\end{definition}

Several aspects of our solution concepts can be explained by the following example. 

\begin{example}\label{ex5}
	Let $C=\R^2_+$ and let $F: \R \rightrightarrows \R^2$ be defined by 
		$$ \gr F = \cone \left\{ \begin{pmatrix}1 \\0 \\ 0 \end{pmatrix},\; \begin{pmatrix} \phantom{-}1 \\ \phantom{-}2 \\-1 \end{pmatrix} \right\}.$$
		Then for $x \in \dom F = \R_+$ we have
		$$ F(x) = \conv\left\{ \begin{pmatrix} 0 \\ 0 \end{pmatrix},\; \begin{pmatrix} 2 x \\-x \end{pmatrix} \right\}.$$
		Since $\gr F$ is a cone, we have $F=G$.
\end{example}

We easily see that there is no minimal point and hence no solution in the sense of Definition \ref{def_min_sol_0} for Example \ref{ex5}. This is remarkable insofar as in vector linear programming existence of a solution already follows from Assumption \ref{ass1}, which is satisfied in Example \ref{ex5}.

Condition \eqref{finiteinfimizer3} motivates another definition of minimizing points by using the ordering $\preccurlyeq_{0^+\P}$ and minimizing directions by using the ordering $\preccurlyeq_C$. 
For a vector ordering, minimality with respect to the (``bigger'') cone $0^+\P$ implies minimality with respect to the (``smaller'') cone $C \subseteq 0^+\P$. This is not necessarily true for the set ordering we use here. This means that in vector linear programming the vector ordering induced by $C$ can be replaced by the vector ordering induced by $0^+\P$ without loosing the property of minimality with respect to the original cone $C$. Moreover, in vector linear programming, under Assumption \ref{ass1}, a solution based on $0^+\P$-minimizing points and $C$-minimizing directions always exists. Thus the following definition can be seen as another possibility to generalize the solution concept of vector linear programming to the set-valued case.

\begin{definition}\label{def_min_sol_1}
	An element $\bar{x} \in \dom F$ is called a {\em minimizing point} for \eqref{P} if 
	\begin{align*}
		F(x) \preccurlyeq_{0^+\P} F(\bar x) \quad \Rightarrow \quad F(x) + 0^+\P = F(\bar x) + 0^+\P.
	\end{align*}
	A nonzero element $\hat{x} \in \dom G$ is called a {\em minimizing direction} for \eqref{P} if
	\begin{align*}
		G(x) \preccurlyeq_C G(\hat{x}) \quad \Rightarrow \quad G(x) + C = G(\hat{x}) +  C.
	\end{align*}
	A finite infimizer $(\bar{X},\hat{X})$ is called a {\em solution} to \eqref{P} if all elements of  $\bar{X}$ are minimizing points and all elements of $\hat{X}$ (if any) are minimizing directions.
\end{definition}

With Definition \ref{def_min_sol_1}, $\bar x = 0$ is a minimizing point in Example \eqref{ex5} and for $\bar X = \{0\}$ we have
\begin{equation}\label{eq_1566}
	\P = \conv \bigcup_{x \in \bar X} F(x) + 0^+\P . 
\end{equation}	
More generally, under Assumption \ref{ass1}, there always exists a finite set $\bar X$ of minimizing points such that \eqref{eq_1566} holds. This follows from the existence result in \cite{paper} for bounded problems (using $0^+\P$ as ordering cone). While Definition \ref{def_min_sol_1} is adequate with respect to minimizing points, minimizing directions can cause problems. In Example \ref{ex5}, we have
	 $G(\alpha x) + C \supsetneq G(\beta x) +C$ whenever $\alpha > \beta$. Thus a minimizing direction as defined in Definition \ref{def_min_sol_1} does not exist. Consequently, also a solution in the sense of Definition \ref{def_min_sol_1} does not necessarily exist.

Let us discuss another variation of the solution concept. The second condition in \eqref{finiteinfimizer3} can be expressed equivalently as
	$$ 0^+\P = \cl\cone \bigcup_{x \in \hat X} \cl\cone G(x) + C.$$
This motivates the following definition.

\begin{definition}\label{def_min_sol_2}
	An element $\bar{x} \in \dom F$ is called a {\em minimizing point} for \eqref{P} if 
	\begin{align*}
		F(x) \preccurlyeq_{0^+\P} F(\bar x) \quad \Rightarrow \quad F(x) + 0^+\P = F(\bar x) + 0^+\P.
	\end{align*}
	A nonzero element $\hat{x} \in \dom G$ is called a {\em minimizing direction} for \eqref{P} if
	\begin{align*}
		\cl\cone G(x) \preccurlyeq_C \cl\cone G(\hat{x}) \quad \Rightarrow \quad \cl\cone G(x) + C = \cl\cone G(\hat{x}) +  C.
	\end{align*}
	A finite infimizer $(\bar{X},\hat{X})$ is called a {\em solution} to \eqref{P} if all elements of  $\bar{X}$ are minimizing points and all elements of $\hat{X}$ (if any) are minimizing directions.
\end{definition}

According to Definition \ref{def_min_sol_2}, $(\bar X,\hat X)=(\{0\},\{1\})$ is a solution of the problem in Example \ref{ex5}. However we find another example where a solution does not exist (even though Assumption \ref{ass1} is fulfilled).

\begin{example}\label{ex6}
	Let $C=\R^2_+$ and let $F: \R^2 \rightrightarrows \R^2$ be defined by 
		$$ \gr F = \cone \left\{ \begin{pmatrix}\phantom{-}1 \\\phantom{-}0 \\ \phantom{-}2 \\-1 \end{pmatrix},\;
		                          \begin{pmatrix} 1 \\ 0 \\0 \\ 0 \end{pmatrix},\;
								   \begin{pmatrix} \phantom{-}0 \\  \phantom{-}1 \\ -1 \\  \phantom{-}2 \end{pmatrix},\;
									\begin{pmatrix} 0 \\ 0 \\1 \\ 0 \end{pmatrix},\;
								     \begin{pmatrix} 0 \\ 0 \\0 \\ 1 \end{pmatrix}\right\}.$$
Since $\gr F$ is a cone, we have $F=G$. 
For $x \in \R^2$ of the form $x = x_\alpha = (\alpha,\;1-\alpha)^T$, $\alpha \geq 0$, we have
$$ G(x_\alpha)=\conv \left\{ \begin{pmatrix} \alpha-1 \\ 2-2\alpha \end{pmatrix},\; \begin{pmatrix} 3 \alpha -1 \\ 2-3\alpha \end{pmatrix}\right\} + C$$
For $0 \leq \alpha \leq \beta < 1$ we have 
$$ \cl\cone G(x_\beta) \preccurlyeq_C \cl\cone G(x_\alpha)$$
with
$$ \cl\cone G(x_\beta) + C \neq \cl\cone G(x_\alpha) + C$$
for sufficiently large $\alpha$.
Moreover, for all $0 \leq \alpha < 1$ we have
$$ \begin{pmatrix}
	-1 \\ \phantom{-}2
\end{pmatrix} \in \cl\cone G(x_\alpha) + C$$
but for $\alpha=1$
$$ \begin{pmatrix}
	-1 \\ \phantom{-}2
\end{pmatrix} \not\in \cl\cone G(x_\alpha)+ C.$$
Thus, there is no minimizing direction.
\end{example}

Another problem with Definition \ref{def_min_sol_2} is that the mapping $x \mapsto \cl\cone G(x)$ is not necessarily convex. A solution method analogous to the one in \cite{paper,paper1} is therefore not possible.

\section{Conclusions}

We defined finite infimizers for not necessarily bounded polyhedral convex set optimization problems. Finite infimizers can be computed by a reformulation of the problem into a vector linear program. This part works analogous to the bounded case in \cite{paper, paper1}.

We defined three types of minimizers and thus three types of solutions. All these variants led to problems with the existence of solutions. These problems are new in comparison with vector linear programming. We conclude that a generalization of the concepts and results of \cite{paper, paper1} to the unbounded case is more involved and requires further investigation. These studies could also touch the fundamental question on adequate requirements for a universal solution concept for set-valued optimization problems in the framework of the complete lattice approach. We believe that such a universal solution concept must involve a satisfactory treatment of the simplest 
subclasses of set optimization problems, in particular, of the polyhedral convex set optimization problem. One can accept that existence of solutions is more involved and not always possible in a set-valued framework (in contrast to a vector-valued framework). But one could also try to replace minimality by other properties which generalize the vector-valued case and additionally maintain the existence of solutions under analogous assumptions.

\vskip 6mm
\noindent{\bfseries Acknowledgments}
\vskip 2mm
\noindent   
The authors thank Frank Heyde and Andreas H. Hamel for a very fruitful discussion about the issues on solution concepts addressed in the last section. They also thank two anonymous reviewers for their valuable comments.


\end{document}